\documentclass{amsart}
\usepackage{amsmath}
\usepackage{amssymb}
\usepackage{amsthm}
\usepackage[msc-links,abbrev]{amsrefs}
\usepackage[noabbrev,capitalize]{cleveref}
\newtheorem{theorem}{Theorem}[section]
\newtheorem{lemma}[theorem]{Lemma}
\newtheorem{claim}[theorem]{Claim}

\theoremstyle{definition}

\theoremstyle{remark}
\newtheorem{remark}[theorem]{Remark}

\numberwithin{equation}{section}

\usepackage{orcidlink}
\usepackage{amsfonts}
\usepackage{hyperref}

\begin{document}
\title[Renormalized Area]{Renormalized Area for Minimal Hypersurfaces of 5D Poincar\'e-Einstein Spaces}
\author{Aaron J. Tyrrell}
\address{18A Department of Mathematics and Statistics \\ Texas Tech University \\ Lubbock \\ TX 79409 \\ USA}
\email{aatyrrel@ttu.edu}
\begin{abstract}
In this paper we derive a Gauss-Bonnet formula for the renormalized area of Graham-Witten minimal hypersurfaces of 5-dimensional Poincar\'e-Einstein spaces. The formula we derive expresses the renormalized area in terms of integrals of pointwise scalar Riemannian invariants. We also prove a result which gives a characterization of minimal hypersurfaces with $L^2$ second fundamental form in terms of conformal geometry at infinity.
\end{abstract}
\maketitle
\section{Introduction}
\label{Intro}
The renormalized volume of Poincar\'e-Einstein manifolds, $(X^{n+1},g_+),$ and the renormalized area of their minimal submanifolds has been a focus of much study over the past 20 years in both geometry and physics. Renormalized volume arises via a construction carried out by Henningson and Skenderis \cite{firstrenvol} in the physics literature and was developed by Graham \cite{graham1999volume} in the math literature. It is defined by taking the order-zero term in the expansion in $\epsilon$ of the quantity $\text{Vol}_{g_+} ({r > \epsilon}),$ where $r$ is a special defining function for $\partial X$ (section 2 will cover all preliminaries and definitions in detail). An important formula is Anderson's \cite{Anderson} Gauss-Bonnet formula for the renormalized volume, $V(g_+),$ of a 4D Poincar\'e-Einstein space which states that
\begin{equation}\label{anderson}
	8\pi^2\chi(X^4) = 6V(g_+)+\frac{1}{4}\int_X|W_{g_+} |^2_{g_+}dv_{g_+}.
\end{equation}
Here $W_{g_+}$ is the Weyl tensor of $g_+.$ Since $|W_{g_+} |^2_
{g_+}$ is a pointwise conformal invariant of weight -4, the integral is guaranteed to converge despite the infinite volume of $(X, g_+)$. In his paper, Anderson notes that given a Poincar\'e-Einstein manifold $(X^4, g_+)$ \eqref{anderson} implies the following rigidity result:
\begin{equation}\label{rigidity}
	V(g_+)\leq \frac{4\pi^2}{3}\chi(X^4),
\end{equation}
with equality if and only if $g_+$ is hyperbolic.\par Renormalized area was introduced by Graham and Witten in \cite{GW} where they outlined a theory of area renormalization for a certain class of minimal submanifolds of Poincar\'e-Einstein spaces; we call these "Graham-Witten" minimal submanifolds. The focus of this paper is Graham-Witten minimal hypersurfaces of a Poincar\'e-Einstein space $(X^5,g_+).$ These are complete minimal hypersurfaces, $(Y^4,h_+),$ dividing $X$ into two pieces $X^-$ and $X^+$ such that $Y\cap \partial X=\Sigma^3 \neq \emptyset.$ Renormalized area has been studied in both geometry and physics. It has been applied to the asymptotic Plateau problem, which is the problem of finding a minimal submanifold of hyperbolic space bounding a pre scribed asymptotic boundary at infinity. This situation was treated conclusively by Anderson in \cite{anderson1982complete}, \cite{anderson1983complete}. This problem can be posed on more general spaces including Poincar\'e-Einstein spaces. Renormalized area proves to be a useful tool here for generalizing many techniques from the compact case to these minimal submanifolds extending to infinity. There is a rich existence theory developed by Oliveira and Moret \cite{de1998complete}, see also Coskunuzer \cite{coskunuzer2003minimal}, \cite{coskunuzer2006generic}. A survey paper on this problem was also written by Coskunuzer \cite{coskunuzer2009asymptotic}. Some regularity results for minimal hypersurfaces in this setting are given by \cite{hard},\cite{lin},\cite{linn},\cite{tonegawa} (an error in \cite{linn} is corrected in \cite{tonegawa}). The renormalized area functional is also relevant in the Ads\slash CFT correspondence. Maldacena proposed in \cite{phys1} that the expectation value of the Wilson loop operator corresponding to a closed loop $\gamma$ in $\partial X,$ within the context of the large t'Hooft coupling regime should be given by the renormalized area of a minimal surface bounding $\gamma.$ This is also studied in \cite{phys2}. Both of these motivated \cite{GW} which in turn was part of the motivation for the author of this paper. A physical context in which higher dimensional renormalized area, and in particular the case we consider for this paper, proves to be important is in calculating entanglement entropies for spatial regions within holographic QFTs. This was first observed by Ryu-Takayanagi \cite{ryu2006aspects}, \cite{ryu2006holographic}. They showed that given a holographic CFT on the boundary geometry at infinity of an AdS, the entanglement entropy of a spatial region within the CFT is equal to the renormalized area of a minimal hypersurface of the bulk AdS up to a constant factor. This was generalized further by Hubeny-Rangamani-Takayanagi \cite{bhattacharya2015entanglement}. Graham and Witten showed that renormalized area is well-defined for even-dimensional submanifolds of Poincar\'e-Einstein spaces. Alexakis and Mazzeo \cite{Alex} studied this functional on properly embedded, minimal surfaces of Poincar\'e-Einstein spaces. The work presented in this paper was partly motivated by work from \cite{Alex}. In particular the explicit formula they obtained for the renormalized area, $A(Y),$ of a minimal 2 dimensional hypersurface Y within a 3 dimensional hyperbolic space:
\begin{equation}\label{AM}
A(Y)=-2\pi\chi(Y)-\frac{1}{2}\int_{Y}|\mathring{II}|^2 dA,
\end{equation}
where $\mathring{II}$ is the trace-free second fundamental form of $Y.$ \newline\indent An important property of renormalized area$\backslash$volume is that, when the (sub)manifold in question is even dimensional it is an invariant of that (sub)manifold. Otherwise it will depend on the choice of special defining function used in the construction. Therefore the natural case to consider in order to generalize \eqref{AM} is the following: 
\begin{theorem}[\textbf{Gauss-Bonnet Formula for Renormalized Area}]\label{1.1}
Let $(X^5,g_+)$ be a Poincar\'e-Einstein space. Let $(Y^4,h_+)$ be a Graham-Witten minimal hypersurface and set $A(Y)$ as the renormalized area of $Y.$ Then
\begin{align}\label{best formula}6A(Y)=8\pi^2\chi(Y)&-\int_{Y}\frac{|W|^2_{h_+}}{4}dA_{h_+}+\int_{Y}\frac{|E|^2_{h_+}}{2}dA_{h_+}\\\nonumber &-\int_{Y}\frac{|\mathring{B}|_{h_+}^4}{24}dA_{h_+}-\int_{Y}\frac{\Delta^Y |\mathring{B}|_{h_+}^2}{2}+|\mathring{B}|_{h_+}^2dA_{h_+},
 \end{align}
 where $W$ is the Weyl tensor on $Y,$ $E$ is the trace-free Ricci tensor on $Y,$ and $\mathring{B}$ is the trace-free second fundamental form of $Y.$ Note that we are also claiming that every term on the right-hand side above is finite.
\end{theorem}
The following theorem is interesting because it gives a characterization of minimal hypersurfaces with an $L^2$ second fundamental form in terms of conformal geometry at infinity. 
\begin{theorem}[\textbf{Holographic Characterization of $L^2$ Second Fundamental Forms}]\label{1.2}
Let the setup be the same as in \cref{1.1}. Then either 
\begin{enumerate}
\item $\partial Y$ is umbilic in $\partial X,$ which is equivalent to  
\begin{equation*}
\int_Y \Delta^Y|\mathring{B}|^2dA_{h_+}=0, 
\end{equation*}
which is equivalent to $|\mathring{B}|\in L^2(Y).$
\newline \newline Or\newline
\item $\partial Y$ is not umbilic in $\partial X$ which is equivalent to the integral above diverging to $-\infty.$
\end{enumerate}
\end{theorem}
\begin{remark}
Renormalized area can be defined for a more general class of hypersurfaces than what we refer to as Graham-Witten hypersurfaces above. These results will also hold for those hypersurfaces. These can be defined as complete, non-compact, minimal hypersurfaces of $(X,g_+)$ which are asymptotic to a hypersurface contained in the boundary of $X,$ such that under compactification they can be represented locally near their asymptotic boundary as graphs of functions with sufficiently regular polyhomogeneous expansions.
\end{remark}
This article is organized as follows:\newline\indent
In \cref{Intro} we have already stated the main results of the paper.
In \cref{Prelim} we cover preliminaries consisting of a brief introduction to Poincar\'e-Einstein manifolds and renormalized volume and area. We only cover enough background so as to make this paper essentially self-contained and to set the conventions, notation and coordinate systems that will be used for the rest of the paper. For a more extensive treatment of renormalized volume$\backslash$area please see \cite{graham1999volume} and \cite{GW}. In \cref{Main Proof} we state two preliminary results that can be thought of as stepping stones towards the main results of this paper. We also show that these preliminary results imply the main results. \cref{3.1 Proof} is the most technically involved portion of the paper, this is where we prove the first stepping stone result stated in \cref{Main Proof}; \cref{3.1}. In \cref{Proof 3.2} we prove the second such result; \cref{3.2}. \cref{Remarks} consists of some final remarks and thoughts of the author as this paper comes to a close. \cref{Ack} contains acknowledgements.

\section{Preliminaries}   
\label{Prelim}
A complete Riemannian metric $g_+$ on the interior $\mathring{X}^{n+1}$ of a compact manifold with boundary $X$ is said to be conformally compact if $\overline{g} := r^2g_+$ extends as a metric to $X$, where $r$ is a defining function for $M = \partial X.$ By a defining function for $M$ we mean  \begin{align*} r &> 0 \hspace{2mm}\text{on $\mathring{X},$} \\
  r &= 0 \hspace{2mm}\text{on $M,$}\\ dr &\neq 0 \hspace{2mm}\text{on $M.$}
 \end{align*} We call $[g_+]_{\infty}:=[\overline{g}|_{M}]$ the conformal infinity of $(X,g_+),$ where $[\overline{g}|_{M}]$ is the conformal class of $\overline{g}|_{M}.$ This is clearly a well-defined invariant of $g_+$ as $r^2g_+|_{M}$ is conformal to $s^2g_+|_{M}$ for any two defining functions $r$ and $s.$ A less obvious invariant of $g_+$ is $|dr|_{r^2g_+}^2\Big{|}_M$ which is independent of defining function $r.$ Conformally transforming the Riemann tensor of $g_+=:g$ yields
 \begin{equation}\label{1}
 R_{ijkl}^{g_+}=-|dr|^2_{\overline{g}}(g_{ik}g_{jl}-g_{il}g_{jk})+O_{ijkl}(r^{-3}),    
 \end{equation}
 where $O_{ijkl}(r^{-3})$ is the component function of a covariant 4-tensor and is $O(r^{-3})$ as $r\to 0.$ Notice that this implies that the tensor corresponding to those components vanishes with respect to the $g_+-\text{norm}$ as $r\to 0$ since $(g_+)^{ij}=r^2\overline{g}^{ij}.$ 
 This implies that $Sec_p\to -|dr|^2_{r^2g_+}$ as $p\to M.$ \par
  Asymptotically hyperbolic manifolds are complete manifolds that have asymptotic sectional curvatures equal to $-1.$  A special defining function for $M$ is a defining function $r$ such that  $|dr|_{\overline{g}}^2\equiv 1$ in a neighborhood of $M$; rather than just on $M$ itself. It is a lemma of Graham's \cite{graham1999volume} that for every metric $g_{\infty}$ in the conformal infinity of an asymptotically hyperbolic metric $g_+$ there is a special defining function, $r$, such that $r^2g_+|_M=g_{\infty}.$ A conformally compact manifold that also satisfies the Einstein condition, $\text{Ric}(g_+)=-ng_+$ is called Poincar\'e-Einstein. For such a manifold, it can be checked by contracting \eqref{1} that $|dr|_{\overline{g}}^2\Big{|}_M=1.$ It follows that  
 Poincar\'e-Einstein manifolds are special cases of asymptotically hyperbolic manifolds. \par A defining function determines for some $r_0$ an identification of $M\times [0,r_0)$ with a neighborhood of $M$ in $X:$ $(p,t)\in M\times [0,r_0)$ corresponding to $\phi(p,t),$ where $\phi$ is the flow of $\nabla_{\overline{g}}r.$ If  $r$ is a special defining function then $r(\phi(p,t))=t.$ So we can think about the $t$ coordinate as just being $r$ and $\nabla
_{\overline{g}}r$ is orthogonal to the slices $M\times \{t\}.$ The metric $\overline{g}$ then takes the form $$\overline{g}=g_r+dr^2$$ where $g_r$ is a 1-parameter family of metrics on $M$ (note that in our notation this implies $g_0=g_{\infty}$). Using the Einstein condition, one can compute what is known as the Fefferman-Graham expansion for $g_r.$ When $n$ is odd, this takes the form:\newline
$$g_r=g_{0}+g^{(2)}r^2+\text{(even powers)}+g^{(n-1)}r^{n-1}+g^{(n)}r^n+O(r^{n+1}),$$   
\newline
where $g^{(i)}$ are tensors on $M$ that are formally determined for $1\leq j \leq n-1,$ and $g^{(n)}$ is trace-free with respect to any metric in $[\overline{g}|_{M}]$. In $M\times [0,r_0)_r$ we can write the volume form of $g_+$ as
\begin{equation*}
   d\nu_{g_+}=r^{-n-1} \bigg(\frac{\mathrm{det}g_r}{\mathrm{det} g_0}\bigg)^{1\slash  2} d\nu_{g_0} dr.
\end{equation*}
Then using Jacobi's formula 
\begin{equation*}
    \frac{d}{dt}(det A(t))= (det A(t))tr(A^{-1}(t)A'(t))
\end{equation*}
coupled with the locally determined terms of the Fefferman-Graham expansion allows us to compute
\begin{equation}
\bigg(\frac{\mathrm{det}g_r}{\mathrm{det}g_0}\bigg)^{1\slash  2}=1+\nu^{(2)}r^2+\text{(even terms)}+\nu^{(n)}r^n+O(r^{n+1})
\end{equation}
where $\nu^{(j)}$ are determined functions on $M$ with $\nu^{(n)}=0$ if $n$ is odd, which we now assume. Next we consider the asymptotics of $\mathrm{Vol}_{g_{+}}(r>\epsilon).$ Let $C=\mathrm{Vol}_{g_{+}}(r>r_0),$ then 
\begin{align}
\mathrm{Vol}_{g_+}(r>\epsilon)&= \int_{r >\epsilon}d\nu_{g_+}\\\nonumber
&=C+\int_{\epsilon}^{r_0}\int_{M}[1+\nu^{(2)}r^2+\text{(even terms)}+\nu^{(n)}r^n+O(r^{n+1})]r^{-n-1}d\nu_{g_0} dr\\\nonumber &=c_0\epsilon^{-n}+c_2\epsilon^{-2+n}+\text{(odd powers)}+c_{n-1}\epsilon^{-1}+V+O(\epsilon).
\end{align}
An essential property of the renormalized volume $V$ is
that it does not depend on this choice of special defining function. This is generally not true if $X$ is odd-dimensional; see \cite{graham1999volume} for the proof. The argument in \cite{graham1999volume} shows that for $n$ odd the renormalized volume of $g_+$ is well-defined but since renormalized volume is a non-local invariant, finding an explicit expression is non-trivial. 
\par
Graham and Witten \cite{GW} outlined an analogous theory to that of renormalized volume but for a class of minimal submanifolds of Poincar\'e-Einstein manifolds. Here we sketch how that outline looks in the particular setting that we will be in for this paper. Let $(X^5,g_+)$ be a Poincar\'e-Einstein space and let $(Y^4,h_+)$ be a complete minimal hypersurface dividing $X$ into two pieces $X^-$ and $X^+$ such that $Y\cap \partial X=\Sigma^3 \neq \emptyset$ (this is what we mean by a Graham-Witten minimal hypersurface).  Let $(\Sigma,k_{\infty})$ be the boundary of $Y$ with $k_{\infty}$ being the induced metric. Given a choice of special defining function $r$ for $\partial X$ let $\{x^{\mu}\}$ be coordinates on a neighborhood $U_0$ in $M$ that intersects $\Sigma$ non-trivially, in particular choose coordinates such that $\Sigma\cap U_0=\{x^4=0 \}$ and $|\partial_4|_{g_{\infty}}\equiv 1.$ This can be done by writing $M=M^-\cup M^+$ where $M^-\cap M^+=\Sigma$ and defining
\begin{equation}
  x^4(q)= 
  \begin{cases}
  \hspace{2mm}d_{g_0}(q,\Sigma) &\text{if $q\in U_0\cap M^+$ } \\
   -d_{g_0}(q,\Sigma) &\text{if $q\in U_0\cap M^-$ }.
  \end{cases}
\end{equation}
Extend these coordinates to a local system on $X$ via the flow of $\nabla_{\overline{g}}r.$ Call this extended neighborhood $U.$ Because $r$ is a special defining function for $M$ we know that this coordinate extension is $\{x^i,r \}.$ We let $0\leq i,j \leq 4,$ $1\leq \mu,\lambda \leq 4$ and $ 1\leq a,b \leq 3.$ We let $r$ correspond to the index $0$ and $0\leq \alpha ,\beta \leq 3.$ Now we let $z\colon \mathbb{R}^4\to \mathbb{R} $ be a smooth function such that $Y\cap U=\{p\in U : x^4(p)=z(x^a(p),r(p))\}.$ This induces a natural local parametrization for $Y,$  \begin{equation} \phi(x^a,r):=(x^a, z(x^a,r),r).
\end{equation}
Note that our tuples are ordered such that the $0^{th}$ coordinate comes last. 
$\phi$ induces coordinates $\{ \hat{x}^{\alpha} \}=\{\hat{x}^{a}, \hat{r} \}=\{x^{a}|_Y, r|_Y \}$  on $Y.$ The induced coordinate basis for $TY$ is given by $\phi_{*}(\partial_{\alpha})=\partial_{\hat{\alpha}}.$ We will omit the hat from coordinate functions when the context is clear but not from the basis vectors. From now on, when we write a Greek index on an induced metric, we will mean this basis i.e. \begin{align}
    h_{\alpha\beta}&:=g(\phi_{*}(\partial_{\alpha}), \phi_{*}(\partial_{\beta})).
\end{align} 
Furthermore, whenever an index attached to an object on $TY$ appears without a hat, then we are suppressing the hat notation. So for example if $B$ is the scalar second fundamental form of $Y$ then $ ``B(\partial_{\alpha},\partial_{\beta})”$ is not defined so by $B_{\alpha\beta}$ we mean to refer to the object viewed with respect to the basis $\{\phi_{*}(\partial_{\alpha}), \phi_{*}(\partial_{r})\}=\{\partial_{\hat{\alpha}}, \partial_{\hat{r}}\}$ i.e. $B_{\alpha\beta}=B(\partial_{\hat{\alpha}},\partial_{\hat{\beta}}).$ 
It is straightforward to compute
\begin{equation*}
    \phi_{*}(\partial_{\alpha})=\partial_{\alpha}+z_{\alpha}\partial_{4},
\end{equation*}
\begin{equation*}
h_{\alpha\beta}=g_{\hat{\alpha}\hat{\beta}}=g_{\alpha\beta}+g_{4\alpha}z_{\beta}+g_{4\beta }z_{\alpha}+g_{44}z_\alpha z_{\beta},
\end{equation*}
which gives
\begin{align}\label{h}\overline{h}_{ab}&=\overline{g}_{ab}+\overline{g}_{4a}z_{b}+ \overline{g}_{4b}z_{a} +\overline{g}_{44}z_{a}z_{b},\\\nonumber
 \overline{h}_{ar}&=\overline{g}_{4a}z_{r}+\overline{g}_{44}z_{a}z_{r},\\\nonumber
 \overline{h}_{r r}&=1+\overline{g}_{44}z_{r}z_{r}.
 \end{align}
Now consider the minimal surface equation on 
$V_{\epsilon}^{r_0}:=\phi^{-1}(U\cap Y\cap\{r_0>r>\epsilon \}),$ for $\epsilon\in (0,r_0)$ then we get
\begin{equation*}\delta \int_{V^{r_0}_{\epsilon}} \sqrt{\mathrm{det}\hspace{0.5mm} h} dxdr=\frac{1}{2} \int_{V_{\epsilon}^{r_0}} \sqrt{\mathrm{det} \hspace{0.5mm}h}h^{\alpha\beta}\delta h_{\alpha\beta} dxdr=0.
\end{equation*}
The formula for the variation of the induced metric in our collar neighborhood is given by 
\begin{align*}
    \delta h_{\alpha\beta}=&\big[g_{\alpha\beta,4}+g_{4\alpha,4}z_{,\beta}+g_{4\beta,4}z_{,\alpha}+g_{44,4}z_{,\alpha}z_{,\beta} \big]\delta z\\\nonumber
    &\hspace{0mm}+g_{4\alpha}\delta z_{,\beta}+g_{4\beta}\delta z_{,\alpha}+g_{44}z_{,\alpha}\delta z_{,\beta}+g_{44}z_{,\beta}\delta z_{,\alpha}.
\end{align*}
It follows that if $Y$ is a critical point of the area functional then 
\begin{equation*}
    \bigg[\partial_{\beta}+\frac{1}{2}(\log(\mathrm{det}\hspace{0.5mm} h))_{,\beta} \bigg] \bigg[ h^{\alpha\beta}\bigg(g_{\alpha 4}+g_{44}z_{,\alpha} \bigg) \bigg]-\frac{1}{2}h^{\alpha\beta}\bigg[g_{\alpha\beta,4}+2g_{\alpha 4,4}z_{,\beta}+g_{44,4}z_{,\alpha}z_{,\beta} \bigg]=0.
\end{equation*}
If we expand the metric in this equation in terms of the $rr$ terms, mixed $ar$ terms and terms that only involve $ab$ and let $L=\log(\mathrm{det}\hspace{0.5mm}h)$ and drop the comma notation for derivatives of functions when there is no ambiguity, we get 
\begin{align*}
&\hspace{4mm}\big[\partial_{b}+\frac{1}{2}L_{b}\big]\big[\overline{h}^{rb}\overline{g}_{44}z_{r}+\overline{h}^{ab}(\overline{g}_{a 4}+\overline{g}_{44}z_{a})\big] \\\nonumber
&+\big[\partial_{r}-4\frac{1}{r}+\frac{1}{2}L_r\big]\big[\overline{h}^{rr}\overline{g}_{44}z_{r}+\overline{h}^{a r}(\overline{g}_{a 4}+\overline{g}_{44}z_{a})\big] \\\nonumber
&-\frac{1}{2}\overline{h}^{ab}\big[\overline{g}_{ab,4}+2\overline{g}_{a 4,4}z_{b}+\overline{g}_{44,4}z_{a}z_{b}\big] \\\nonumber
&-\overline{h}^{a r}\big[\overline{g}_{a 4,4}z_{r}+\overline{g}_{44,4}z_{a}z_{r} \big] \\\nonumber
&-\frac{1}{2}\overline{h}^{rr}\big[\overline{g}_{44,4}z_{r}z_{r} \big]=0.
\end{align*}
From this equation, one can derive
\begin{equation}\label{z}
  z(x^a,r)=\frac{\eta}{6}r^2+z^{(4)}(x^a)\frac{r^4}{4!}+O(r^5),
\end{equation}
where $\eta$ is the mean curvature of $\Sigma$ in $M$ and $z^{(4)}$ is a locally determined function on $\Sigma.$ 
Then using Jacobi's formula 
\begin{equation*}
    \frac{d}{dr}(\mathrm{det}\hspace{0.5mm} h)= (\mathrm{det}\hspace{0.5mm} h)h^{\alpha\beta}h'_{\alpha\beta}
\end{equation*}
along with \eqref{h} and letting $h_0=\overline{h}|_{T\Sigma}$ we can get a local expression for the area form 
\begin{equation*}
dA_{h_+}=r^{-4}\sqrt{\frac{\mathrm{det}\hspace{0.5mm} \overline{h}}{\mathrm{det}\hspace{0.5mm} h_0}} ds_{k_{\infty}} dr
\end{equation*}
\begin{equation}
\bigg(\frac{\mathrm{det}\hspace{0.5mm} \overline{h}}{\mathrm{det}\hspace{0.5mm} h_0}\bigg)^{1\slash  2}=1+\alpha^{(2)}r^2+O(r^{4}).
\end{equation}
Where $\alpha^{(2)}$ is a formally determined function on $\Sigma$.
If we define $Y_{\epsilon}:=Y\cap \{r\geq\epsilon \}$ and let $C=\mathrm{Area}_{h_+}(Y_{r_0})$ then this yields the expansion 
\begin{align}
\mathrm{Area}_{h_+}\big(Y_{\epsilon}\big)&=\int_{Y_{\epsilon}} dA_{h_+}\\\nonumber
&=C+\int_{\epsilon}^{r_0}\int_{\Sigma}[1+\alpha^{(2)}r^2+O(r^4)]r^{-4}ds_{k_{\infty}} dr\\\nonumber
&=c_{0}\epsilon^{-3}+c_1\epsilon^{-1}+A+O(\epsilon).
\end{align}
Where $c_0$ and $c_1$ are integrals over $\Sigma$ of locally determined functions and the constant term is an invariant of $h_+$ that we call the renormalized area of $h_+;$ see \cite{GW} for a proof.
\section{Proof of Main Results}
\label{Main Proof}
Let $g_{\infty}$ be a representative for the conformal infinity of $g_+$ such that $\partial Y$ is minimal in $\partial X$ with respect to $g_{\infty}$ and let $r$ be a special defining function for $\partial X$ such that $\overline{g}:=r^2g_+$ satisfies $\overline{g}|_{TM}=g_{\infty}.$ This $r$ exists since it can be obtained by solving the following equation: given $g_0\in [g_+]_{\infty},$ let $s$ be a function on $M$ such that

\begin{equation*}
    \big(H-(\nabla_{g_0} \log s)^{\perp} \big)=0
\end{equation*}
where $H$ is mean curvature of $\partial Y$ as a submanifold of $(M,g_0)$. By the conformal transformation law for mean curvature, $s^2g_0$ would have $0$ mean curvature on $\partial Y$ as a submanifold of $M.$
Then by Graham's Lemma there exists a unique special defining function $r$ giving us \begin{equation*}
    r^2g_+|_{TM}=s^2g_0.
\end{equation*}
We call $r$ a special-special defining function or special${}^2$ defining function.
With this choice made, we first we prove the following intermediate result:
\begin{theorem}\label{3.1}
 Let the setup be the same as in \cref{1.1} and choose a special${}^{\textit{2}}$ defining function $r$ as outlined above. Then \begin{align}\label{res}6A(Y)=8\pi^2\chi(Y)&-\int_{Y}\frac{|W|^2_{h_+}}{4}dA_{h_+}+\int_{Y}\frac{|E|^2_{h_+}}{2}dA_{h_+}\\\nonumber&-\int_{Y}\frac{|\mathring{B}|_{h_+}^4}{24}dA_{h_+}-f.p.\int_{Y_\epsilon}|\mathring{B}|_{h_+}^2dA_{h_+}
 \end{align}
where $Y_{\epsilon}=Y\cap \{r\geq\epsilon\}$ and $f.p.\int_{Y_{\epsilon}}|\mathring{B}|_{h_+}^2dA_{h_+}$ is the finite part as $\epsilon\to 0$ of $\int_{Y_{\epsilon}}|\mathring{B}|_{h_+}^2dA_{h_+}.$
\end{theorem}  \noindent Note: The above formula actually holds for an arbitrary special defining function $r$ but \cref{3.1} as stated is a more straightforward stepping stone to the main results.
\par The strategy for proving this is to begin with the Chern-Gauss-Bonnet formula for 4-dimensional manifolds with boundary applied to $(\overline{Y},\overline{h}).$ We proceed by conformally transforming all of the terms in correspondence to the change $\overline{h}\to h_+.$ One of these terms then gives us a constant, 6; this is where the $6A(Y)$ term in our Gauss-Bonnet formula ultimately comes from. We then analyze the asymptotics of each term as $\epsilon \to 0$ and arrive at the result in \cref{3.1}.
The next result we need is one that gives us an explicit expression for $f.p.\int_{Y_{\epsilon}}|\mathring{B}|_{h_+}^2dA_{h_+};$ this can be thought of as the renormalized $L^2(Y)$ norm of the second fundamental form of $Y.$
\begin{theorem}\label{3.2}
Let the setup be the same as that in \cref{1.1} and choose a special${}^{\textit{2}}$ defining function $r$ as outlined above. Then
\begin{equation}
|\int_{Y}\frac{\Delta^Y |\mathring{B}|_{h_+}^2}{2}+|\mathring{B}|_{h_+}^2dA_{h_+}|<\infty
\end{equation}
and
\begin{equation}\label{B expansion}
    \int_{Y_{\epsilon}}|\mathring{B}|_{h_+}^2dA_{h_+}=\epsilon^{-1}\oint_{\partial Y} |\mathring{II}|^2_{k_{\infty}} ds_{k_{\infty}}+\int_{Y}\frac{\Delta^Y |\mathring{B}|_{h_+}^2}{2}+|\mathring{B}|_{h_+}^2dA_{h_+}+O(\epsilon),
\end{equation}
where $\mathring{II}$ is the trace-free second fundamental form of $\partial Y$ in $M$ with respect to the compactified metric restricted to the ambient boundary and $k_{\infty}$ is this metric restricted to the boundary of $Y.$ 
\end{theorem}
 \noindent Note: The above formula actually holds for an arbitrary special defining function $r.$\par This result is based on the observation that although the $L^2(Y)$ norm of $\mathring{B}$ is generally infinite, the term in its asymptotic expansion as $\epsilon\to 0$ which blows up is actually equivalent by Stoke's theorem to $$-\frac{1}{2}\int_{Y_{\epsilon}}\Delta^Y |\mathring{B}|_{h_+}^2 dA_{h_+}$$
 and furthermore, the above integral has no constant term in its asymptotic expansion. In other words
 \begin{equation}\label{ress}
 \int_{Y_{\epsilon}}\Delta^Y |\mathring{B}|_{h_+}^2 dA_{h_+}=-\frac{2}{\epsilon}\oint_{\partial Y} |\mathring{II}|^2_{k_{\infty}} ds_{k_{\infty}}+O(\epsilon)
 \end{equation}
 and
 \begin{equation}
 \int_{Y_{\epsilon}}|\mathring{B}|_{h_+}^2dA_{h_+}=\epsilon^{-1}\oint_{\partial Y} |\mathring{II}|^2_{k_{\infty}} ds_{k_{\infty}}+f.p. \int_{Y_{\epsilon}}|\mathring{B}|_{h_+}^2dA_{h_+}+O(\epsilon).
 \end{equation}
This implies that 
\begin{equation}
|\int_{Y} \frac{\Delta^Y |\mathring{B}|_{h_+}^2}{2}+|\mathring{B}|_{h_+}^2 dA_{h_+}| <\infty
\end{equation}
and furthermore, that 
\begin{equation}\label{fpp}
f.p.\int_{Y_{\epsilon}}|\mathring{B}|_{h_+}^2dA_{h_+}=\int_{Y} \frac{\Delta^Y |\mathring{B}|_{h_+}^2}{2}+|\mathring{B}|_{h_+}^2 dA_{h_+}.
\end{equation}
Plugging \eqref{fpp} into \eqref{res} gives \cref{1.1}. To see how \cref{1.2} follows from the above results, consider \eqref{B expansion}:
\begin{equation*}
    \int_{Y_{\epsilon}}|\mathring{B}|_{h_+}^2dA_{h_+}=\epsilon^{-1}\oint_{\partial Y} |\mathring{II}|^2_{k_{\infty}} ds_{k_{\infty}}+\int_{Y}\frac{\Delta^Y |\mathring{B}|_{h_+}^2}{2}+|\mathring{B}|_{h_+}^2dA_{h_+}+O(\epsilon).
\end{equation*}
Taking the limit as $\epsilon\to 0,$ we see that $\underset{\epsilon\to 0}{\lim}\int_{Y_{\epsilon}}|\mathring{B}|_{h_+}^2dA_{h_+} $ is finite if and only if $\mathring{II}\equiv 0,$ i.e. if and only if $\partial Y$ is umbilic in $\partial X.$ Now considering \eqref{ress}, 
it's clear that $\underset{\epsilon\to 0}{\lim}\int_{Y_{\epsilon}}\Delta^Y|\mathring{B}|_{h_+}^2dA_{h_+}=0$ if and only if $\partial Y$ is umbilic in $\partial X$ and that the integral diverges to $-\infty$ otherwise. The next sections are dedicated to the proofs of \cref{3.1} and \cref{3.2}.
\section{Proof of Theorem 3.1}
\label{3.1 Proof}
\subsection{Outline of Approach}The strategy for proving \cref{3.1} is to make a choice of special${}^{2}$ defining function $r$ and to write down the Chern-Gauss-Bonnet formula \cite{chen2009conformal} for $Y\cap \{r>\epsilon\}$ in terms of the compactified metric $\overline{h}=r^2h_+;$ 
\begin{equation*}
8\pi^2\chi(Y_{\epsilon})=\int_{Y_{\epsilon}}\frac{|\overline{W}|^2_{\overline{h}}}{4}+4\sigma_2(\overline{h}^{-1}\overline{P})dA_{\overline{h}}+\oint_{\Sigma_{\epsilon}}\overline{S}_{r}ds_{\overline{k}_{\epsilon}}.
\end{equation*}
Then we conformally transform every term to express them in terms of the singular metric $h_+.$ A constant term then appears in the integrand which gives us a way to obtain a formula for the renormalized area. We know that the left hand-side of the above equation is constant with respect to $\epsilon$ so it must equal the zero-order term of the expansion of the right-hand side. We then pick out the constant terms in the asymptotic expansion of every term in the Chern-Gauss-Bonnet formula after conformally transforming the formula. \subsection{Proof}We put a bar over a metric dependent quantity to indicate that it is taken with respect to the compactified metric. Some objects we will decorate with either a subscript or superscript of the metric they are taken with respect to depending on what is more convenient. In this context
\begin{align*}
X&=(X,g_+),\\ \overline{X}&=(\overline{X},\overline{g}),\\ Y&=(Y,h_+),\\ \overline{Y}&=(\overline{Y},\overline{h}).
\end{align*}In calculations involving indices we use $g_{ij}:=(g_{+})_{ij}.$
\noindent
Now if we take arbitrary coordinates $\{y^{\alpha}\}$ on $Y$ and extend the chart induced basis to $X$ by the unit normal $\mu_Y,$ and let the index $N$ correspond to $\mu_Y$ and let Latin indices run over $\{y^{\alpha},\mu_Y \}$, we see   
\begin{align*}
 -20&=R_X\\
 &=g^{ij}g^{kl} R_{ikjl}^X \\
 &=2g^{ij}R^X_{iNjN}+g^{\alpha\beta}g^{\gamma\delta}R^X_{\alpha\gamma\beta\delta }\\
 &=2Ric_X(\mu_Y,\mu_Y)+	g^{\alpha\beta}g^{\gamma\delta} R^X_{\alpha\gamma\beta\delta}\\
 &=-8g(\mu_Y,\mu_Y)+ g^{\alpha\beta}g^{\gamma\delta} R^X_{\alpha\gamma\beta\delta}\\
&=-8+g^{\alpha\beta}g^{\gamma\delta}R^X_{\alpha\gamma\beta\delta}. \end{align*}
 But the Gauss equation gives us
 
\begin{equation}
\hspace{-7mm}R^X_{\alpha\gamma\beta\delta}=R^Y_{\alpha\gamma\beta\delta}+B_{\alpha\delta}B_{\gamma\beta}-B_{\alpha\beta}B_{\gamma\delta}.
\end{equation}

 \noindent
 Therefore, \begin{align*}
 \hspace{16mm}-20&=-8+g^{\alpha\beta}g^{\gamma\delta}R_{\alpha\gamma\beta\delta}^X \\
 &=-8+	g^{\alpha\beta}g^{\gamma\delta}[R^Y_{\alpha\gamma\beta\delta}+B_{\alpha\delta}B_{\gamma\beta}-B_{\alpha\beta}B_{\gamma\delta}]\\
 &=-8+R_Y+|B|^2_{h_+}-\big[tr_{h_+}B\big]^2\\
 &=-8+R_Y+|B|^2_{h_+}.
\end{align*}
This give us 
\begin{equation*}
\hspace{-30mm}-12=R_Y+|B|^2_{h_+},
\end{equation*}
isolating the curvature term and squaring both sides gives us
\begin{equation}\label{SC}
\hspace{-18mm}(R_Y)^2=|B|^4_{h}
+24|B|^2_{h}+144. 
\end{equation}
As stated in \cite{chen2009conformal}, the Chern-Gauss-Bonnet formula for 4-manifolds with boundary can be expressed as
\begin{equation}\label{cgb}
8\pi^2\chi(Y_{\epsilon})=\int_{Y_\epsilon}\frac{|\overline{W}|^2_{h_+}}{4}+   4\sigma_2(\overline{h}^{-1}\overline{P}) dA_{\overline{h}}+\oint_{\Sigma_{\epsilon}}\overline{S}_rds_{\overline{k}_{\epsilon}},
\end{equation}
where $Y_{\epsilon}=Y\cap\{r\geq\epsilon \},$ $\Sigma_{\epsilon}=Y\cap\{r=\epsilon\}$ 
and \begin{equation}\label{S}\overline{S}_r=R^{\overline{Y}}\overline{H}_{r}-2 Ric^{\overline{Y}}(\overline{\nu}_r,\overline{\nu}_r) \overline{H}_{r}-2R_{\overline{Y}}{}^d{}_{adb}\overline{L}_{r}^{ab}+\frac{2}{3}\overline{H}_{r}^3-2\overline{H}_{r}|\overline{L}_{r}|^2_{\overline{h}}+\frac{4}{3}tr(\overline{L}_{r}^3),
\end{equation}
where $\overline{\nu}_r$ is the inward-pointing $\overline{h}$-unit normal to $\Sigma_r$ in $Y.$ $a,b,c,d $ correspond to coordinates on $\partial Y.$  $\sigma_2(\overline{h}^{-1}\overline{P})(p)$ is the second symmetric function of the endomorphism of $T_pY$ given by $V\to \overline{P}_{\alpha}^{\beta} V^{\alpha}$
and where $\overline{P}$ is the Schouten tensor of $\overline{h};$
\begin{equation*}
    \overline{P}=\frac{1}{2}\bigg(Ric^{\overline{Y}}-\frac{R^{\overline{Y}}}{6}\overline{h} \bigg).
\end{equation*}
And where $\overline{L}_{\epsilon}$ is the second fundamental form of $\Sigma_{\epsilon}$ within $Y_{\epsilon}$ with respect to the inward pointing unit-normal and $\overline{H}_{\epsilon}$ is the trace of $\overline{L}_{\epsilon}$ with respect to $\overline{k}_{\epsilon}:=$ $\overline{h}|_{\Sigma_{\epsilon}}.$
We will first write $4\sigma_2(\overline{h}^{-1}\overline{P})$ in terms of $4\sigma_2(h^{-1}_+P_+)$ ($P_+$ is the Schouten tensor associated to $h_+$). Then we will use
\begin{align}\label{sig}
4\sigma_2(h_+^{-1}P_+)&=-\frac{|Ric^{Y}|^2_{h_+}}{2}+\frac{R^2_{Y}}{6}\\\nonumber
\text{and}\\
\label{sig'}|Ric^{Y}|^2_{h_+}&=|E|^2_{h_+}+\frac{R_{Y}^2}{4}
\end{align}
in the pursuit of our preliminary formula for the renormalized area of $Y$ which we denote by $A(Y).$
\begin{lemma}\label{sigma lemma}
\begin{align}\label{sym}
4\sigma_2(\overline{h}^{-1}\overline{P})&=4r^{-4}\sigma_2(h_+^{-1}P_+)\\\nonumber&+2\nabla_{\overline{h}}^{\hat{\alpha}}(r^{-3}|dr|^2_{\overline{h}}r_{\hat{\alpha}}-\frac{r_{\hat{\alpha}}\Delta_{\overline{h}}r}{r^2}+r^{-2}r_{\hat{\alpha}\hat{\beta}}r^{\hat{\beta}}+r^{-1}R^{\overline{Y}}_{\hat{\alpha} \hat{\beta}}r^{\hat{\beta}}-\frac{1}{2}r^{-1}R^{\overline{Y}}r_{\hat{\alpha}})
\end{align}
(The Greek index on a scalar function here corresponds to covariant differentiation on $(\overline{Y},\overline{h})$ and $\Delta_{\overline{h}}=\nabla^{\overline{h}}_{\hat{\alpha}}\nabla_{\overline{h}}^{\hat{\alpha}}$).
\end{lemma}
\begin{proof}
\begin{equation}\nonumber
\sigma_2(h_+^{-1}P_+)=\frac{1}{2}\big([tr_{h_+}P_+]^2-|P_+|_{h_+}^2\big). 
\end{equation}
Let $\varphi=-\log r.$ Then the conformal transformation formula for $P_+$ gives us
\begin{equation}\nonumber
P^+_{\alpha\beta}=\overline{P}_{\alpha\beta}-\varphi_{\hat{\alpha}\hat{\beta}}+\varphi_{\hat{\alpha}}\varphi_{\hat{\beta}}-\frac{1}{2}|d\varphi|^2_{\overline{h}}\overline{h}_{\alpha\beta}
\end{equation}
\begin{equation}\nonumber
P^+_{\alpha\beta}=\overline{P}_{\alpha\beta}+\frac{r_{\hat{\alpha}\hat{\beta}}}{r}-\frac{1}{2r^2}|dr|^2_{\overline{h}}\overline{h}_{\alpha\beta}.
\end{equation}
Taking the trace of both sides gives
\begin{equation}\nonumber
    tr_{h_+} P^+=r^2tr_{\overline{h}}\overline{P}+r\Delta_{\overline{h}}r-2|dr|^2_{\overline{h}}.
\end{equation}
This allows us to compute the conformal change of $\sigma_2(h_+^{-1}P_+)$
\begin{align}\nonumber
    &4\sigma_2(h_+^{-1}P_+)=2\big( [tr_{h_+} P_+]^2-|P_+|^2_{h_+}\big)=\\\nonumber
    &=2\big(\big[r^2[tr_{\overline{h}} \overline{P}]+r\Delta_{\overline{h}}r-2|dr|_{\overline{h}}^2 \big]^2\\\nonumber
    &\hspace{4mm}-r^4\big[\overline{P}_{\alpha\beta}+\frac{r_{\alpha\beta}}{r}-\frac{1}{2r^2}|dr|^2_{\overline{h}}\overline{h}_{\alpha\beta} \big]\big[\overline{P}^{\alpha\beta}+\frac{r^{\alpha\beta}}{r}-\frac{1}{2r^2}|dr|^2_{\overline{h}}\overline{h}^{\alpha\beta} \big] \big)\\\nonumber
    &=2\big(\big[[r^4tr_{\overline{h}}\overline{P}]^2+2r^3tr_{\overline{h}}\overline{P}\Delta_{\overline{h}}r-4r^2tr_{\overline{h}}\overline{P}|dr|^2_{\overline{h}}+r^2(\Delta _{\overline{h}}r)^2-4r\Delta_{\overline{h}}r|dr|^2_{\overline{h}}+4|dr|^4_{\overline{h}} \big] \\\nonumber
    &\hspace{5mm}-r^4\big[|\overline{P}|^2+\frac{2r^{\alpha\beta}}{r}\overline{P}_{\alpha\beta}-\frac{tr_{\overline{h}}\overline{P}}{r^2}|dr|^2_{\overline{h}}+\frac{1}{r^2}|\nabla^2_{\overline{h}}r|^2_{\overline{h}}-\frac{\Delta_{\overline{h}}r}{r^3}|dr|^2_{\overline{h}}+\frac{1}{r^4}|dr|^4_{\overline{h}} \big]\big)\\\nonumber
    &=2r^4\big[[tr_{\overline{h}}\overline{P}]^2-|\overline{P}|^2_{\overline{h}} \big]+4r^3tr_{\overline{h}}\overline{P}\Delta_{\overline{h}}r-6r^2tr_{\overline{h}}\overline{P}|dr|^2_{\overline{h}}\\\nonumber
&\hspace{5mm}+2r^2(\Delta_{\overline{h}}r)^2-6r(\Delta_{\overline{h}}r)|dr|^2_{\overline{h}}+8|dr|^4_{\overline{h}}-4r^{\alpha\beta}\overline{P}_{\alpha\beta}r^3-2r^2|\nabla^2_{\overline{h}}r|^2_{\overline{h}}+6|dr|^4_{\overline{h}}\\\nonumber
&=4r^4\sigma_2(\overline{h}^{-1}\overline{P})+6|dr|^4_{\overline{h}}-6r(\Delta_{\overline{h}}r)|dr|^2_{\overline{h}}+2r^2[(\Delta_{\overline{h}}r)^2-|\nabla^2_{\overline{h}}r|^2_{\overline{h}}-3tr_{\overline{h}}\overline{P}|dr|^2_{\overline{h}}]\\\nonumber
    &\hspace{5mm}+4r^3\big[tr_{\overline{h}}\overline{P}\Delta_{\overline{h}}r-\overline{P}_{\alpha\beta}r^{\alpha\beta} \big]. \\\nonumber
\end{align}
Therefore \begin{align}\label{sym''}
4r^{-4}\sigma_2(h_+^{-1}P_+)=&\\\nonumber
&4\sigma_2(\overline{h}^{-1}\overline{P})+6r^{-4}|dr|^4_{\overline{h}}-6r^{-3}(\Delta_{\overline{h}}r)|dr|^2_{\overline{h}}\\\nonumber&+2r^{-2}[(\Delta_{\overline{h}}r)^2-|\nabla^2_{\overline{h}}r|^2_{\overline{h}}-3tr_{\overline{h}}\overline{P}|dr|^2_{\overline{h}}]\\\nonumber &+4r^{-1}\big[tr_{\overline{h}}\overline{P}\Delta_{\overline{h}}r-\overline{P}_{\alpha\beta}r^{\alpha\beta} \big].
\end{align}
Expanding each term on the right-hand side of \eqref{sym} gives 
\begin{align}
\nabla^{\hat{\alpha}}_{\overline{h}}(r^{-3}|dr|^2_{\overline{h}} r_{\hat{\alpha}})&=-3r^{-4}|dr|^4_{\overline{h}}+r^{-3}(r^{\hat{\beta}\hat{\alpha}}r_{\hat{\beta}}r_{\hat{\alpha}}+r^{\hat{\beta}}r_{\hat{\beta}}{}^{\hat{\alpha}}r_{\hat{\alpha}}),\\
\nabla^{\hat{\alpha}}_{\overline{h}}(-r^{-2}r_{\hat{\alpha}}\Delta^{\overline{h}} r)&=-r^{-2}(\Delta^{\overline{h}} r)^2-r^{-2}r_{\hat{\alpha}}r^{\hat{\beta}}{}_{\hat{\beta}}{}^{\hat{\alpha}}+2r^{-3}|dr|_{\overline{h}}^2\Delta^{\overline{h}} r,\\\label{eq}
\nabla^{\hat{\alpha}}_{\overline{h}}(r^{-2}r_{\hat{\alpha}\hat{\beta}}r^{\hat{\beta}})&=-2r^{-3}r^{\hat{\alpha}}r_{\hat{\alpha}\hat{\beta}}r^{\hat{\beta}}+r^{-2}r_{\hat{\alpha}\hat{\beta}}{}^{\hat{\alpha}}r^{\hat{\beta}}+r^{-2}r_{\hat{\alpha}\hat{\beta}}r^{\hat{\beta}\hat{\alpha}},\\\label{eq2}
\nabla^{\hat{\alpha}}_{\overline{h}}(r^{-1}R_{\hat{\alpha}\hat{\beta}}r^{\hat{\beta}})&=-r^{-2}r^{\hat{\alpha}}r^{\hat{\beta}}R_{\hat{\alpha}\hat{\beta}}^{\overline{h}}+r^{-1}\nabla^{\hat{\alpha}}_{\overline{h}}R_{\hat{\alpha}\hat{\beta}}^{\overline{h}}r^{\hat{\beta}}+r^{-1}R_{\hat{\alpha}\hat{\beta}}^{\overline{Y}}r^{\hat{\beta}\hat{\alpha}},\\
\nabla^{\hat{\alpha}}_{\overline{h}}(-\frac{1}{2}r^{-1}R^{\overline{Y}}r_{\hat{\alpha}})&=\frac{1}{2} r^{-2}R^{\overline{Y}}|dr|_{\overline{h}}^2-\frac{1}{2} r^{-1}R_{\overline{Y}}{}^{\hat{\alpha}}r_{\hat{\alpha}}-\frac{1}{2} r^{-1}R^{\overline{Y}}\Delta^{\overline{h}} r.
\end{align}
Using the Ricci identity to commute the outer covariant derivatives of the middle factor of the $r^{-2}r_{\hat{\alpha}\hat{\beta}}{}^{\hat{\alpha}}r^{\hat{\beta}}$ term in \eqref{eq} followed by the contracted Bianchi identity on the second term of \eqref{eq2}, then writing the Ricci factor in the last term of \eqref{eq2} in terms of the Schouten tensor and comparing to \eqref{sym''} it becomes clear that \eqref{sym} is true.
\end{proof}
Plugging \eqref{sig} and \eqref{sig'} into \eqref{sym} gives 
\begin{align}\label{sym'}
4\sigma_2(\overline{h}^{-1}\overline{P})=&r^{-4}(-\frac{|E|_{h_+}^2}{2}+\frac{R_{\overline{Y}}^2}{24})\\\nonumber &+2\nabla_{\overline{h}}^{\hat{\alpha}}(r^{-3}|dr|^2_{\overline{h}}r_{\hat{\alpha}}-\frac{r_{\hat{\alpha}}\Delta_{\overline{h}}r}{r^2}+r^{-2}r_{\hat{\alpha}\hat{\beta}}r^{\hat{\beta}}+r^{-1}R^{\overline{Y}}_{\hat{\alpha} \hat{\beta}}r^{\hat{\beta}}-\frac{1}{2}r^{-1}R^{\overline{Y}}r_{\hat{\alpha}}) \end{align}
and using the above equality to re-write $4\sigma_2(\overline{h}^{-1}P)$ in  \eqref{cgb} we get
\begin{align}\nonumber\label{cpct'''}
8\pi^2\chi(Y_{\epsilon})&=\int_{Y_{\epsilon}}\frac{|\overline{W}|^2_{\overline{h}}}{4}dA_{\overline{h}}+\int_{Y_{\epsilon}}r^{-4}(\frac{R_{Y}^2}{24}-\frac{|E|_{h_+}^2}{2})dA_{\overline{h}}\\\nonumber&+2\int_{Y_{\epsilon}}\nabla_{\overline{h}}^{\hat{\alpha}}(\hat{r}^{-3}|d\hat{r}|^2_{\overline{h}}\hat{r}_{\hat{\alpha}}-\frac{\hat{r}_{\hat{\alpha}}\Delta_{\overline{h}}\hat{r}}{\hat{r}^2}+\hat{r}^{-2}\hat{r}_{\hat{\alpha}\hat{\beta}}\hat{r}^{\hat{\beta}}+\hat{r}^{-1}R^{\overline{Y}}_{\hat{\alpha} \hat{\beta}}\hat{r}^{\hat{\beta}}-\frac{1}{2}\hat{r}^{-1}R^{\overline{Y}}\hat{r}_{\hat{\alpha}}) dA_{\overline{h}}\\\nonumber &\hspace{5mm}+\oint_{\Sigma_{\epsilon}}\overline{S}_rds_{\overline{k}_{\epsilon}}.
\end{align}
Using the conformal invariance of the Weyl term and using $\hat{r}^{-4}dA_{\overline{h}}=dA_+$ yields:
\begin{align}\nonumber
8\pi^2\chi(Y_{\epsilon})&=\int_{Y_{\epsilon}}\frac{|W|^2_{h_+}}{4}-\frac{|E|^2_{h_+}}{2}+\frac{R^2_{Y}}{24}dA_{h_+}\\\nonumber
&+\int_{Y_{\epsilon}} 2\overline{\nabla}^{\hat{\alpha}}(r^{-3}|dr|^2_{\overline{h}}r_{\hat{\alpha}} -\frac{r_{\hat{\alpha}}\Delta_{\overline{h}}r}{r^2}+r^{-2}r_{\hat{\alpha}\hat{\beta}}r^{\hat{\beta}} +r^{-1}R^{\overline{Y}}_{\hat{\alpha} \hat{\beta}}r^{\hat{\beta}}-\frac{1}{2}r^{-1}R^{\overline{Y}}r_{\hat{\alpha}}) dA_{\overline{h}}\\\nonumber
&+\oint_{\Sigma_{\epsilon}}\overline{S}_{r}ds_{\overline{k}_{\epsilon}}.\end{align}
Using integration by parts on the second integral and the fact that \begin{equation}\overline{\nu}_{r}=\frac{\nabla_{\overline{h}}r}{|\nabla_{\overline{h}}r|}=\frac{\nabla_{\overline{h}}r}{\sqrt{\overline{h}^{r r}}}=\partial_{\hat{r}}+O^{\hat{\alpha}}(r^5)\partial_{\hat{\alpha}}\end{equation}   yields
\begin{align*}
8\pi^2\chi(Y_{\epsilon})&=   \int_{Y_{\epsilon}}\frac{|W|^2_{h_+}}{4}-\frac{|E|^2_{h_+}}{2}+\frac{R^2_Y}{24}dA_{h_+}\\\nonumber &-\int_{\Sigma_{\epsilon}}  2(r^{-3}|dr|^2_{\overline{h}}r_{\hat{r}}  -\frac{r_{\hat{r}}\Delta_{\overline{h}}r}{r^2}   +r^{-2}r_{\hat{\alpha} \hat{r}}r^{\hat{\alpha}}+r^{-1}R^{\overline{Y}}_{\hat{r}\hat{\beta}}r^{\hat{\beta}}-\frac{1}{2}r^{-1}R^{\overline{Y}}r_{\hat{r}})ds_{\overline{k}(\epsilon)}\\\nonumber &+O(\epsilon^2)+\oint_{\Sigma{\epsilon}}\overline{S}_{r}ds_{\overline{k}_{\epsilon}}.
\end{align*}
Defining $$\mathcal{B}_{r}:=\overline{S}_{r}-2 (r^{-3}|dr|^2_{\overline{h}} 
-\frac{\Delta_{\overline{h}}r}{r^2}+r^{-2}r_{\hat{r}\hat{\beta}}r^{\hat{\beta}}+r^{-1}R^{\overline{Y}}_{\hat{r}\hat{\beta}}r^{\hat{\beta}}-\frac{1}{2}r^{-1}R^{\overline{Y}})$$ allows us to write
\begin{align}
8\pi^2\chi(Y_{\epsilon})&=\int_{Y_{\epsilon}}\frac{|W|^2_{h_+}}{4}-\frac{|E|^2_{h_+}}{2}+\frac{R^2_Y}{24}dA_{h_+} +\oint_{\Sigma_{\epsilon}} \mathcal{B}_r ds_{\overline{k}_{\epsilon}}+O(\epsilon^2). \end{align}
Now we plug \eqref{SC} in for $R_Y^2$ and get 
\begin{align}\label{GB}8\pi^2\chi(Y_{\epsilon})
&=\int_{Y_{\epsilon}}\frac{|W|^2_{h_+}}{4}dA_{h_+}-\int_{Y_{\epsilon}}\frac{|E|^2_{h_+}}{2}dA_++6\int_{Y_{\epsilon}}dA_+\\\nonumber &\hspace{5mm}+\int_{Y_{\epsilon}}\frac{|B|_{h_+}^4}{24}+|B|_{h_+}^2dA_+\\\nonumber &\hspace{5mm}+\oint_{\Sigma_{\epsilon}}\mathcal{B}_{r} ds_{\overline{k}_{\epsilon}}+O(\epsilon^2).\end{align}$\underset{\epsilon\to 0}{\lim}\int_{Y_{\epsilon}}|W|^2_{h_+}dA_+$ clearly converges since $|W|^2_{h_+}$ is a conformal invariant of weight $-4.$ $\chi(Y_{\epsilon})=\chi(Y)$ for $\epsilon$ small enough since we can use the flow of $\nabla_{\overline{h}}r$ to deformation retract $Y$ onto $Y_{\epsilon}.$ Next we need a technical lemma to allow us to identify the constant term in the asymptotic expansion of the right-hand side of \eqref{GB}. We will get that the trace-free Ricci term and the $|B|^4$ term converge and we will obtain some formulas that will allow us to determine the finite part of the boundary term.
\begin{lemma}\label{4.2}
\begin{align}\label{Order}
&|dr|^2_{\overline{h}}=1+O(r^5),\\
&\overline{S}_r=O(r),\\\label{R}
&\partial_{\hat{r}} R^{\overline{Y}}(\overline{\nu}_r,\overline{\nu}_r)|_{r=0}=0,\\\label{R'}
&\partial_{\hat{r}} R^{\overline{Y}}|_{r=0}=0,\\
&|\int_{Y}|E|^2_{h_+} dA_+|<\infty,\\
&|\int_{Y}|B|^4_{h_+} dA_+|<\infty.
\end{align}
\end{lemma}
\begin{proof}
Now we are using the same notational setup and coordinates described in the preliminary section. Recall \eqref{z}
\begin{equation*}
  z(x^a,r)=\frac{\eta}{6}r^2+z^{(4)}(x^a)\frac{r^4}{4!}+O(r^5),
\end{equation*}
where $\eta$ is the mean curvature of $\Sigma$ in $M.$ 
Our choice of special defining function gives $\eta\equiv 0$ and therefore
\begin{equation*}
z=O(r^4).
\end{equation*}
It follows that \begin{align}
 \overline{h}_{\alpha \beta}=\overline{g}_{\alpha\beta}+O_{\alpha\beta}(r^5)	
 \end{align}
and
\begin{align}\label{h inverse}
\overline{h}^{\alpha\beta}=\overline{g}^{\alpha\beta}+O^{\alpha\beta}(r^5).	
\end{align}
Now,
\begin{equation}\label{dr}
|dr|^2_{\overline{h}}=\overline{h}^{\alpha\beta}r_{\hat{\alpha}}r_{\hat{\beta}}=\overline{h}^{rr}=1+O(r^5).
\end{equation}
So we get the first part of our lemma (in fact $\overline{h}^{rr}=1+O(r^6)$ but the line above is all we need). 
Recall \eqref{S}:  
\begin{align*}\overline{S}_{r}=R^{\overline{Y}}\overline{H}_{\hat{r}}-2Ric^{\overline{Y}}(\overline{\nu}_{r},\overline{\nu}_{r})\overline{H}_{r}-2R_{\overline{Y}}{}^d{}_{adb}\overline{L}_{r}^{ab}+\frac{2}{3}\overline{H}_{r}^3-2\overline{H}_{r}|\overline{L}_{r}|^2_{\overline{h}}+\frac{4}{3}tr(\overline{L}_{r}^3).\end{align*}
All of the terms involving $\overline{H}^{r}$ and $\overline{L}^{r}$ vanish on the boundary since the second fundamental form of $\Sigma$ within $\overline{Y}$ vanishes i.e. $\Sigma$ is totally geodesic in $(\overline{Y},\overline{h})$, to see this consider the following:
Recall that $\overline{\nu}_{\epsilon}$ is the inward pointing unit-normal to $\partial Y_{\epsilon}$ within $Y_{\epsilon}.$

\begin{align}\overline{\nu}_{r}=\frac{\nabla_{\overline{h}}r}{|\nabla_{\overline{h}}r|}=\frac{\nabla_{\overline{h}}r}{\sqrt{\overline{h}^{r r}}}&=\partial_{\hat{r}}+O^{\hat{\alpha}}(r^5)\partial_{\hat{\alpha}}\\\nonumber
&=\partial_{r}+O^{i}(r^3)\partial_i.
\end{align}
Therefore
\begin{align*}\nonumber
\overline{L}_{ab}^{r}&= \overline{h}(\nabla_{\phi_{*}(\partial_a)}^{\overline{Y}}\phi_{*}(\partial_b), \overline{\nu}_{r})\\\nonumber
&=\overline{\Gamma}_{ab}^{\alpha}\overline{h}_{\alpha r}+O(r^3)\\
&=\overline{\Gamma}_{ab}^{c}\overline{h}_{c r}+\overline{\Gamma}_{ab}^r\overline{h}_{rr}+O(r^3)\\
&=\overline{\Gamma}_{ab}^r\overline{h}_{rr}+O(r)
\end{align*}
\begin{align}\label{chris}
\overline{\Gamma}_{ab}^r=\frac{\overline{h}^{rd}}{2}\bigg\{\partial_{\hat{a}} \overline{h}_{bd}+\partial_{\hat{b}} \overline{h}_{ad}-\partial_{\hat{d}} \overline{h}_{ab} \bigg\}+\frac{\overline{h}^{rr}}{2}\bigg\{\partial_{\hat{a}} \overline{h}_{br}+\partial_{\hat{b}} \overline{h}_{ar}-\partial_{\hat{r}} \overline{h}_{ab} \bigg\}=O(r)
\end{align}
hence  
\begin{equation}\label{minimal L}
\overline{L}_{ab}^r=O(r).
\end{equation}
This implies that $\overline{S}_r=O(r).$ 
Now consider the following

\begin{align}
    R^{\overline{Y}}&=\overline{h}^{\alpha\beta}R^{\overline{Y}}_{\alpha\beta}\\\nonumber
    &=\overline{h}^{ab}R^{\overline{Y}}_{ab}+2\overline{h}^{ar}R^{\overline{Y}}_{ar}+\overline{h}^{rr}R^{\overline{Y}}_{rr}\\\nonumber
    &=\overline{h}^{ab}R^{\overline{Y}}_{ab}+R^{\overline{Y}}_{rr}+O(r^2).
\end{align}
Now 
\begin{align}
    \hspace{-4mm}R^{\overline{Y}}_{ab}&=\overline{h}^{\alpha\beta}R^{\overline{Y}}_{\alpha a \beta b}\\\nonumber
    &=\overline{h}^{cd}R^{\overline{Y}}_{cadb}+R^{\overline{Y}}_{rarb}+O(r^2)
\end{align}
Now
\begin{align}
    \hspace{18mm}R^{\overline{Y}}_{cadb}&=R^{\overline{Y}}_{cad}{}^{\alpha}\overline{h}_{\alpha b}\\\nonumber
    &=R^{\overline{Y}}_{cad}{}^{e}\overline{h}_{e b}+R^{\overline{Y}}_{cad}{}^{r}\overline{h}_{r b}\\\nonumber
    &=R^{\overline{Y}}_{cad}{}^{e}\overline{h}_{e b}+O(r^2)\\\nonumber
    &=\big[ \partial_{\hat{a}} \overline{\Gamma}^{e}_{cd}-\partial_{\hat{c}}\overline{\Gamma}^{e}_{ad}+\overline{\Gamma}_{cd}^{\beta}\overline{\Gamma}_{\beta a}^{e}-\overline{\Gamma}_{ad}^{\beta}\overline{\Gamma}_{\beta c}^{e} \big]\overline{h}_{eb}+O(r^2).
\end{align}
Now we compute some Christoffel symbols:
\begin{align}
    \overline{\Gamma}^e_{cd}&=\frac{\overline{h}^{ea}}{2}\bigg\{\partial_{\hat{c}} \overline{h}_{ad}+\partial_{\hat{d}} \overline{h}_{ac}-\partial_{\hat{a}} \overline{h}_{cd} \bigg\}+O(r^2)\\\nonumber
    &=C+O(r^2)\hspace{2mm}\text{($C$ here just means a constant wrt $r;$ i.e. a function on $\Sigma$)} 
\end{align}
\begin{align}
\overline{\Gamma}_{cd}^{\beta}\overline{\Gamma}_{\beta a}^{e}&=\overline{\Gamma}_{cd}^{b}\overline{\Gamma}_{b a}^{e}+\overline{\Gamma}_{cd}^{r}\overline{\Gamma}_{r a}^{e}\\\nonumber
&=\overline{\Gamma}_{cd}^{r}\overline{\Gamma}_{r a}^{e}+C+O(\hat{r}^2)\\\nonumber
&=\frac{1}{2}\bigg\{\partial_{\hat{c}} \overline{h}_{rd}+\partial_{\hat{d}} \overline{h}_{cr}-\partial_{\hat{r}} \overline{h}_{cd} \bigg\}\frac{\overline{h}^{ec}}{2}\bigg\{\partial_{\hat{r}} \overline{h}_{ac}+\partial_{\hat{a}} \overline{h}_{rc}-\partial_{\hat{c}} \overline{h}_{ra} \bigg\}+C+O(r^2)\\\nonumber
&=\frac{1}{4}\big\{ -\partial_{\hat{r}} \overline{h}_{cd}\big\}\big\{ \partial_{\hat{r}} \overline{h}_{ac} \big\}+C+O(r^2)\\\nonumber
&=C+O(r^2).
\end{align}
It follows that 
\begin{equation}
\hspace{-58mm}R^{\overline{Y}}_{cadb}=C+O(r^2)
\end{equation}
\begin{align}
R^{\overline{Y}}_{rarb}&=\big[ \partial_{\hat{a}} \overline{\Gamma}^{e}_{rr}-\partial_{\hat{r}}\overline{\Gamma}^{e}_{ar}+\overline{\Gamma}_{rr}^{\beta}\overline{\Gamma}_{\beta a}^{e}-\overline{\Gamma}_{ar}^{\beta}\overline{\Gamma}_{\beta r}^{e} \big]\overline{h}_{eb}+O(r^2).
\end{align}
Now we compute some Christoffel symbols
\begin{align*}
    \overline{\Gamma}_{rr}^e&=\frac{\overline{h}^{ed}}{2}\big\{\partial_{\hat{r}}\overline{h}_{dr}+\partial_{\hat{r}}\overline{h}_{dr}-\partial_{\hat{d}}\overline{h}_{rr} \big\}+O(r^2)\\
    &=O(r^2).\\
     \partial_{\hat{r}}\overline{\Gamma}_{ar}^e&=\partial_{\hat{r}}\big[\frac{\overline{h}^{ed}}{2}\big\{\partial_{\hat{a}}\overline{h}_{dr}+\partial_{\hat{r}}\overline{h}_{da}-\partial_{\hat{d}}\overline{h}_{rr} \big\}+O(r^5)\big] \\
     &=\partial_{\hat{r}}\big[\big\{\frac{\overline{h}^{ed}\partial_{\hat{r}}\overline{h}_{ad}}{2}\}+O(r^5)\big]\\
     &=\overline{g}^{ed}\frac{\overline{g}''_{ad}}{2}\bigg|_{r=0}+O(r^2) \hspace{2mm}{\text{(this follows since $\overline{g}'$ and $\overline{g}'''$ vanish for $r=0$).}}\\
     \overline{\Gamma}_{rr}^{r}&=\frac{\overline{h}^{rr}}{2}\big\{\partial_{\hat{r}}\overline{h}_{rr}+\partial_{\hat{r}}\overline{h}_{rr}-\partial_{\hat{r}}\overline{h}_{rr} \big\}+O(r^5)\\
     &=O(r^4)\\
     \overline{\Gamma}_{ar}^e&=\frac{\overline{h}^{ed}}{2}\big\{\partial_{\hat{a}}\overline{h}_{dr}+\partial_{r}\overline{h}_{da}-\partial_{\hat{d}}\overline{h}_{rr} \big\}+O(r^5) \\
     &=O(r)
\end{align*}
It follows that $R^{\overline{Y}}_{rarb}$ has vanishing first derivative with respect to $r$ at the boundary. It then follows for $R^{\overline{Y}}_{ab}$ and, once we show it for $R^{\overline{Y}}_{rr},$ it will also follow for $R^{\overline{Y}}.$
So consider 
\begin{align}
    R^{\overline{Y}}_{rr}&=\overline{h}^{\alpha\beta}R^{\overline{Y}}_{\alpha r \beta r}\\\nonumber
    &=\overline{h}^{ab}R^{\overline{Y}}_{a r b r}+R^{\overline{Y}}_{rrrr}+O(r^2)\\\nonumber
    &=\overline{h}^{ab}R^{\overline{Y}}_{a r b r}+O(r^2)\\\nonumber
    &=O(r^2).
\end{align}
So we have the first four parts of our lemma. Now consider the following 
\begin{align}
    E_{\alpha\beta}&=R^Y_{\alpha\beta}-\frac{R^Y}{4}h_{\alpha\beta}\\\nonumber
    &=R^Y_{\alpha\beta}-\frac{-12-|B|^2}{4}h_{\alpha\beta}\\\nonumber
    &=R^Y_{\alpha\beta}+3h_{\alpha\beta}+\frac{|B|^2}{4}h_{\alpha\beta}
\end{align}
Contracting the Gauss equation on one pair of indices gives 
\begin{equation}
    R^{Y}_{\alpha\beta}+B^2_{\alpha\beta}=-3g_{\hat{\alpha}\hat{\beta}}-W^X_{\hat{\alpha}N\hat{\beta}N}.
\end{equation}
Therefore
\begin{align}
 E_{\alpha\beta}&=-B^2_{\alpha\beta}+\frac{|B|^2}{4}h_{\alpha\beta}-W^X_{\hat{\alpha}N\hat{\beta}N}
 \end{align}
 which implies 
 \begin{equation}
 |E|^2=|B^2|^2-\frac{|B|^4}{4}+2(B^2)^{\alpha\beta}W^X_{\hat{\alpha}N\hat{\beta}N}+W^X_{\hat{\alpha}N\hat{\beta}N}W_X^{\hat{\alpha}N\hat{\beta}N}.
 \end{equation}
 Now since Y is minimal, we have $B=\mathring{B}$ and since
 \begin{align}
  \mathring{B}_{\alpha\beta}&=\frac{\mathring{\overline{B}}_{\alpha\beta}}{r}\\
  B^2_{\alpha\beta}&=\mathring{B}_{\alpha\gamma}\mathring{B}_{\delta\beta}h^{\gamma\delta}=\mathring{\overline{B}}_{\alpha\gamma}\mathring{\overline{B}}_{\delta\beta}\overline{h}^{\gamma\delta}=\mathring{\overline{B}}^2_{\alpha\beta}\\
  (\mathring{B}^2)^{\alpha}_{\beta}&=r^2 (\mathring{\overline{B}}^2)^{\alpha}_{\beta}
  \end{align}
 we get that $|\mathring{B}|^4$ and $|\mathring{B}^2|^2$ are pointwise conformal invariants. In particular this gives us that $|B|^2\in L^2(Y).$ Also it is a conformal property of the Weyl tensor that \begin{align}
   W^X_{\hat{\alpha}N}{}^{\hat{\beta}}{}_{N}&=W^{\overline{X}}_{\hat{\alpha}N}{}^{\hat{\beta}}{}_{N}
 \end{align}
 Now since the index $N$ corresponds to a $g_+$ unit-normal to $Y$ and $\overline{N}$ correspond to a $\overline{g}$ unit normal to $Y$ we get 
 \begin{equation}
 W^{\overline{X}}_{\hat{\alpha}N}{}^{\hat{\beta}}{}_{N}=r^2 W^{\overline{X}}_{\hat{\alpha}\overline{N}}{}^{\hat{\beta}}{}_{\overline{N}}.
 \end{equation}
 It now follows that \begin{equation}(B^2)^{\alpha\beta}W^X_{\hat{\alpha}N\hat{\beta}N}\end{equation} and \begin{equation}W^X_{\hat{\alpha}N\hat{\beta}N}W_X^{\hat{\alpha}N\hat{\beta}N}\end{equation} are pointwise conformal invariants of weight $-4$. It then follows that $|E|^2$ is a pointwise conformal invariant and therefore $|E|\in L^2(Y).$ This finishes the proof of the lemma. 
 \end{proof}

Next we show that the boundary integral contributes nothing to the renormalized area.
\begin{claim}\label{4.3}
\begin{equation}
\oint_{\Sigma_{\epsilon}}\mathcal{B}_{r}ds_{\overline{k}_{\epsilon}}=\text{(negative powers of $\epsilon$)}+\text{(positive powers of $\epsilon$)}
\end{equation}
i.e. the above expansion has no finite part as $\epsilon \to 0$.
\end{claim}
\begin{proof}
As $\overline{S}_r=O(r),$ it is enough to show that $$\oint_{\Sigma_{\epsilon}}(r^{-3}|dr|^2_{\overline{h}} 
-\frac{\Delta_{\overline{h}}r}{r^2}+r^{-2}r_{\hat{r}\hat{\beta}}r^{\hat{\beta}}+r^{-1}R^{\overline{Y}}_{\hat{r}\hat{\beta}}r^{\hat{\beta}}-\frac{1}{2}r^{-1}R^{\overline{Y}})ds_{\overline{k}_{\epsilon}}$$
has no finite part as $\epsilon\to 0.$ Because the "length form"  $ds_{\overline{k}_{\epsilon}}$ has no first degree term i.e. $\partial_{\hat{r}}\sqrt{\det{\overline{k}_{r}}}=0$ on $\Sigma,$ it is enough to show that the terms in the integrand have expansions with no $0$-degree term and no $-2$ degree term in $r$, we will call a term with this property "good". Recalling that 
$|dr|^2_{\overline{h}}=1+O(r^5)$ gives us that the first term in the integrand is good. 
Now we compute the second term
\begin{align}
    \Delta_{\overline{h}}r&=\overline{h}^{\alpha\beta}\nabla^{\overline{h}}_{\hat{\alpha}} \nabla^{\overline{h}}_{\hat{\beta}}r\\\nonumber
    &=\overline{h}^{\alpha\beta}\big[\partial_{\hat{\alpha}}r_{\hat{\beta}}-\overline{\Gamma}^{r}_{\alpha\beta} \big]\\\nonumber
    &=-\overline{h}^{\alpha\beta}\overline{\Gamma}^r_{\alpha\beta}\\\nonumber
    &=-\overline{h}^{ab}\overline{\Gamma}^r_{ab}-\overline{h}^{rr}\overline{\Gamma}^r_{rr}+O(r^5).
\end{align}
Recall that $\overline{\Gamma}_{rr}^r=O(r^4)$
and observe 
\begin{align}
    \hspace{14mm}\overline{\Gamma}_{ab}^r&=\frac{1}{2}\big\{\partial_{\hat{a}}\overline{h}_{br}+\partial_{\hat{b}}\overline{h}_{ar}-\partial_{\hat{r}}\overline{h}_{ab} \big\}+O(r^5)\\\nonumber
    &=Cr+O(r^3).
\end{align}
It follows that 
\begin{equation}
    \hspace{-30mm}\frac{\Delta_{\overline{h}}r}{r^2}=\frac{C}{r}+O(r)
\end{equation}
and so the second term is good.
Now for the third term
\begin{align}
   r_{\hat{r}\hat{\beta}}r^{\hat{\beta}}&=r_{\hat{r}\hat{r}}+O(r^5)\\\nonumber
    &=-\overline{\Gamma}_{rr}^r+O(r^5)\\\nonumber
    &=O(r^3).
\end{align}
So the third term has the form $O(r^3)$ and is therefore good. Similarly it can be shown that the last two terms are good because of \eqref{R} and \eqref{R'}.\end{proof}
\noindent \cref{4.3} combined with \eqref{GB} implies

\begin{align}\label{REN AREA''}6A(Y)&=8\pi^2\chi(Y)-\int_{Y}\frac{|W|^2_{h_+}}{4}dA_{h_+}+\int_{Y}\frac{|E|^2_{h_+}}{2}dA_+\\\nonumber&\hspace{18mm}-\int_{Y}\frac{|\mathring{B}|^4}{24}dA_+-f.p.\int_{Y_\epsilon}|\mathring{B}|^2dA_+
\end{align}
where 
\begin{equation}
\int_{Y_\epsilon}|\mathring{B}|^2dA_+=\text{(negative powers of $\epsilon$)}+f.p.\int_{Y_\epsilon}|\mathring{B}|^2dA_++O(\epsilon).
\end{equation}
This completes the proof of \cref{3.1}. \qed \subsection{Proof Summary}
Starting with 
\begin{equation}
8\pi^2\chi(Y_{\epsilon})=\int_{Y_{\epsilon}}\frac{|\overline{W}|^2_{\overline{h}}}{4}+4\sigma_2(\overline{h}^{-1}\overline{P})dA_{\overline{h}}+\oint_{\Sigma{\epsilon}}\overline{S}_{r}ds_{\overline{k}_{\epsilon}},
\end{equation}
\cref{sigma lemma} allowed us to write $\sigma_2(\overline{h}^{-1}\overline{P})$ in terms of $\sigma_2(h_+^{-1}P_{h_+}).$ This lead to 
\begin{align}\nonumber
8\pi^2\chi(Y_{\epsilon})&=\int_{Y_{\epsilon}}\frac{|W|^2_{h_+}}{4}-\frac{|E|^2_{h_+}}{2}+\frac{R^2_Y}{24}dA_{h_+} +\oint_{\Sigma_{\epsilon}} \mathcal{B}_r ds_{\overline{k}_{\epsilon}}+O(\epsilon^2). \end{align}
Then we plugged \eqref{SC} in for $R_Y^2$ and got 
\begin{align}8\pi^2\chi(Y_{\epsilon})
&=\int_{Y_{\epsilon}}\frac{|W|^2_{h_+}}{4}dA_{h_+}-\int_{Y_{\epsilon}}\frac{|E|^2_{h_+}}{2}dA_++6\int_{Y_{\epsilon}}dA_+\\\nonumber &\hspace{5mm}+\int_{Y_{\epsilon}}\frac{|B|_{h_+}^4}{24}+|B|_{h_+}^2dA_+\\\nonumber &\hspace{5mm}+\oint_{\Sigma_{\epsilon}}\mathcal{B}_{r} ds_{\overline{k}_{\epsilon}}+O(\epsilon^2),\end{align}
where $\mathcal{B}_{r}$ represents boundary terms. By \cref{4.2} we have that the trace-free Ricci term and the $|B|^4$ term both converge. $\underset{\epsilon\to 0}{lim}\int_{Y_{\epsilon}}|W|^2_{h_+}dA_+$ clearly converges since $|W|^2_{h_+}$ is a conformal invariant of weight $-4.$ $\chi(Y_{\epsilon})=\chi(Y)$ for $\epsilon$ small enough since we can use the flow of $\nabla_{\overline{h}}r$ to deformation retract $Y$ onto $Y_{\epsilon}.$ By \cref{4.3} we get that the boundary terms do not contribute to the finite part of the above asymptotic expansion. This then implies
\begin{align}6A(Y)&=8\pi^2\chi(Y)-\int_{Y}\frac{|W|^2_{h_+}}{4}dA_{h_+}+\int_{Y}\frac{|E|^2_{h_+}}{2}dA_+\\\nonumber&\hspace{18mm}-\int_{Y}\frac{|\mathring{B}|^4}{24}dA_+-f.p.\int_{Y_\epsilon}|\mathring{B}|^2dA_+.
\end{align}
Note that the above expansion is defined in terms of a particular choice of special defining function that we made originally, specifically one for which $\eta\equiv 0$ (a special-special defining function). This makes the above formula for the renormalized area of $Y$ unsatisfying to a certain extent because we would prefer to obtain a formula that does not depend on any such choices. This issue is remedied by \cref{3.2}.
\section{Proof of Theorem 3.2}
\label{Proof 3.2}
\subsection{Outline of Approach}The strategy for proving \cref{3.2} is to compute the asymptotic expansion of $\int_{Y_{\epsilon}}|B|^2_{h_+}dA_{h_+}$ and of $\frac{1}{2}\int_{Y_{\epsilon}}\Delta^Y|B|^2dA_{h_+}$ and observe that the latter has no constant term and has a term which is singular as $\epsilon\to 0$ that cancels exactly with the singular term in the expansion of the former.

\subsection{Proof}
First we compute a $g$-unit-normal vector field to $Y$ which we call $\mu_Y.$
 Let $w=x^4-z(x^a,r)$ then
 \begin{align}
    \nabla_{g_+}w= g^{ij}[\delta_{i}^4-z_i(x^a,r)]
\partial_j&=g^{4j}\partial_j-g^{ij}z_i\partial_j \\\nonumber
&=g^{4\mu}\partial_{\mu}-g^{a\lambda}z_a\partial_{\lambda}-r^2z_r\partial_r \\\nonumber
&=r^2\partial_4+O(r^4)\partial_4+O^{\lambda}(r^4)\partial_{\lambda}+g^{ab}z_a\partial_b-r^2z_r\partial_r\\\nonumber
&=r^2[1+O(r^2)]\partial_4+O^a(r^4)\partial_a-r^2z_r\partial_r. 
\end{align}
It follows that 
\begin{equation}
    |\nabla_{g_+}w|=r\sqrt{[1+O(r^2)]\overline{g}_{44}+O(r^4)}=r[1+O(r^2)]
\end{equation}
this implies then that 
\begin{equation}
    \frac{1}{|\nabla_{g_+}w|}=\frac{1+O(r^2)}{r}
\end{equation}
Therefore 
\begin{equation}
    \mu_Y=\frac{\nabla_{g_+}w}{|\nabla_{g_+}w|}=r[1+O(r^2)]\partial_4+O^a(r^3)\partial_a-rz_r\partial_r+O(r^6)\partial_r
\end{equation}
hence \begin{align}\label{2}
    \overline{\mu}_Y&=[1+O(r^2)]\partial_4+O^{a}(r^2)\partial_a-z_r\partial_r+O(r^5)\partial_r\\\nonumber
    &=\partial_4+O(r^3)\partial_r+O^{\mu}(r^2)\partial_{\mu}.
\end{align}
Note that $z_r=O(r^3)$ because of our choice of special defining function.
Conformally changing the metric $g_+\to r^2g_+$ then the second fundamental form changes according to 
\begin{align}
    B_{\alpha\beta}&=\frac{\overline{B}_{\alpha\beta}}{r}+\frac{\overline{\mu}_Y(r)}{r^2}\overline{h}_{\alpha\beta}\\\nonumber
    &=\frac{\overline{B}_{\alpha\beta}}{r}+O_{\alpha\beta}(r)
\end{align}
But 
\begin{align}
    \overline{B}_{\alpha\beta}=\overline{g}(\nabla^{\overline{g}}_{\hat{\alpha}}\partial_{\hat{\beta}},\overline{\mu}_Y)
\end{align}
recalling that $\partial_{\hat{\alpha}}=\partial_{\alpha}+z_{\alpha}\partial_4$ and $z=O(r^4)$ gives
\begin{align}\label{B}
    \overline{B}_{\alpha\beta}&=\overline{g}(\nabla^{\overline{g}}_{\alpha}\partial_{\beta},\partial_4)+O(r^2)\\\nonumber
    &=\frac{-\partial_4\overline{g}_{\alpha\beta}}{2}+O(r^2).
\end{align}
Now 
\begin{align}
    |\overline{B}|^2_{\overline{h}}=\overline{h}^{ab}\overline{h}^{cd}\overline{B}_{ac}\overline{B}_{bd}+\overline{B}_{rr}+O(r^5),
\end{align}
but $\overline{B}_{rr}=O(r^3)$ and $\overline{B}_{ab}=\frac{-\partial_4\overline{g}_{ab}}{2}+O(r^2)$ by \eqref{B} so we get
 \begin{equation}
 |\overline{B}|^2_{\overline{h}}=|II|^2_{\infty}+O(r^2).
 \end{equation}
 Then we get
\begin{align}
    |B|^2_{h_+}&=\big[\frac{\overline{B}_{\alpha\beta}}{r}+O_{\alpha\beta}(r)\big]\big[\frac{\overline{B}_{\gamma\delta}}{r}+O_{\gamma\delta}(r)\big]h^{\alpha\gamma}h^{\beta\delta}\\\nonumber
    &=|\overline{B}|^2_{\overline{h}}r^{2}+O(r^4)\\\nonumber
    &=|II|^2_{\infty}r^2+O(r^4)
\end{align}
where $|II|_{\infty}$ corresponds to the norm of the second fundamental form of $\partial Y=\Sigma$ within $\partial X=M$ with respect to the metric $k_{\infty}=\overline{g}\big|_{\partial Y}.$ Recall that $k_{\epsilon}=g\big|_{T\partial Y_{\epsilon}}$ and $ds_{\infty},$ $ds_{k_{\epsilon}}$ are the length forms of $k_{\infty}$ and $k_{\epsilon}$ respectively.  
It follows that 
\begin{equation}
    \int_{Y_{\epsilon}} |B|^2dA_{h_+}=\oint_{\Sigma} |II|^2_{\infty} ds_{\infty}\epsilon^{-1}+f.p.\int_{Y_{\epsilon}} |B|^2dA_{h_+}+O(\epsilon).
\end{equation}
Now we make the following observation:
\begin{equation}
    \int_{Y_{\epsilon}}\Delta^{h_+} |B|^2 dA_{h_+}=-\oint_{\Sigma_{\epsilon}} \nabla_{\nu_{\epsilon}}^{h_+}|B|^2 ds_{k_{\epsilon}}
\end{equation}
where $\nu_{\epsilon}$ is the inward-pointing $h_+$-unit-normal to $\partial Y_{\epsilon}=\Sigma_{\epsilon}$ in $Y_{\epsilon}.$ We can compute $\overline{\nu}_{r}=\frac{\nabla_{\overline{h}}r}{|\nabla_{\overline{h}}r|}$ to get 
\begin{align}
    \nabla_{\overline{h}}r=\overline{h}^{\alpha\beta}r_{\alpha}\partial_{\hat{\beta}}=\overline{h}^{r\beta}\partial_{\hat{\beta}}&=\overline{h}^{rr}\partial_{\hat{r}}+O^{\hat{b}}(r^5)\partial_{\hat{b}}\\\nonumber
    &=\partial_{\hat{r}}+O^{\alpha}(r^5)\partial_{\alpha},
\end{align}
\begin{equation}
    \overline{\nu}_{r}=\partial_{\hat{r}}+O^{\alpha}(r^5)\partial_{\alpha}
\end{equation}
noting the fact that $\nu=r\overline{\nu}$
it follows that 
\begin{align}
    \nabla_{\nu_{r}}|B|^2 &=r\partial_{r} \big[|II|^2_{\infty}r^2+O(r^4) \big]+O(r^5) \\\nonumber
    &=2|II|^2_{\infty}r^2+O(r^4).
\end{align}
Now taking into account the fact that $\partial_{\hat{r}}\sqrt{\mathrm{det}\hspace{0.5mm} \overline{k}_r}$ vanishes on the boundary we see that,
\begin{equation}
-\oint_{\Sigma_{\epsilon}} \nabla_{\nu_{r}}^{h_+}|B|^2 ds_{k_{\epsilon}}= -2\oint_{\Sigma}|II|^2_{\infty} ds_{\infty} \epsilon^{-1}+O(\epsilon)
\end{equation}
therefore we may write 
\begin{equation}
    \int_{Y}|B|^2+\frac{\Delta^{h_+}|B|^2}{2}dA_{h_+}=f.p.\int_{Y_{\epsilon}}|B|^2dA_{h_+}.
\end{equation}
This completes the proof of \cref{3.2}.\qed

\section{Final Remarks}
\label{Remarks}
\noindent Some follow-up questions that arise naturally are the following: \par
Shi-Tsai Feng \cite{feng2016volume} showed that renormalized volume can be defined for asymptotically hyperbolic manifolds of dimension 4 that have a totally geodesic compactification. The Graham-Witten hypersurface we consider in this paper is asymptotically hyperbolic as a manifold in its own right, it also has a totally geodesic compactification. What is the relationship between its renormalized volume in the sense of Feng and its renormalized area. \par In \cite{Alex} the authors obtain a formula for renormalized area for surfaces that aren't minimal but only meet the boundary orthogonally. With this in mind, another direction to take in the 4 dimensional hypersurface case would be to relax conditions on $X$ and $Y$ to the point where $X$ is only asymptotically hyperbolic with a totally geodesic compactification and $Y$ satisfies conformally invariant conditions. We could then obtain a Gauss-Bonnet formula for the renormalized area of $Y$ which will not be an invariant of $Y,$ but will depend on the compactification. The advantage of such a formula is that it will lead to a conformal invariant of $Y.$ The author conjectures that the appropriate conditions on $Y$ should be $z^{(1)}=z^{(3)}=0$ on $\Sigma.$ Many of the terms in the formula in \cref{1.1} are pointwise conformal invariants so this generalization should yield a nice conformal invariant on $Y.$ \par Another question is whether one can bound some of the terms in the Gauss-Bonnet formula in order to obtain a rigidity result on Graham-Witten minimal hypersurfaces. This may be easier to analyze if the formula is first obtained for a more general class of hypersurfaces like those described in the previous paragraph.
\par 
Another direction to take is to generalize the formula in this paper to higher codimensions and to adapt the techniques from \cite{yang2008renormalized} to generalize the formula in this paper to $Y^{2k}$ for all $k\in \mathbb{N}.$ $Y$ is an asymptotically hyperbolic manifold and as such, the results of \cite{graham2003scattering} already apply to $Y.$ 

\section{Acknowledgments}\label{Ack}The majority of this work was carried out during the author's Ph.D. studies at the University of Notre Dame. The author would like to thank his adviser Professor Matthew Gursky for introducing him to this area of geometry and to this problem in particular, as well as for many helpful conversations. The author would also like to thank his postdoctoral supervisor Professor Stephen McKeown for many helpful comments and for help surveying the literature. This work was partially supported by the National University of Ireland \'Eamon de Valera Travelling Studentship in Mathematics. 

\bibliographystyle{amsplain}
\bibliography{bibliography.bib}

\vfill

\end{document}